\documentclass{tran-l}

\usepackage{amssymb}
\usepackage{color}
\usepackage{graphicx}
\usepackage{pgfplots}
\usepackage{xcolor}
\usepackage{amsmath}
\usepackage{enumerate}
\usepackage{bbm}

\usepackage{verbatim}
\newtheorem{theorem}{Theorem}[section]

\theoremstyle{definition}

\newtheorem{example}[theorem]{Example}

\newtheorem{corollary}[theorem]{Corollary}
\theoremstyle{remark}
\newtheorem{remark}[theorem]{Remark}

\numberwithin{equation}{section}

\begin{document}

\title{Extensions of  S-Lemma for Noncommutative Polynomials}


	\author{Feng Guo}
	\address{School of Mathematical Sciences, Dalian  University of Technology,
Dalian, 116024, China}
	\curraddr{}
	\email{fguo@dlut.edu.cn}
	\thanks{}
	
	\author{Sizhuo Yan and  Lihong Zhi}
	\address{Key Lab of Mathematics Mechanization, AMSS, University of Chinese Academy of Sciences, Beijing 100190,
        China}
	\curraddr{}
	\email{yansizhuo@amss.ac.cn, lzhi@mmrc.iss.ac.cn}
	\thanks{This research  is supported by the National Key Research Project of China
2018YFA0306702 (Zhi) and the National Natural Science Foundation of China 12071467 (Zhi).}



\date{}

\dedicatory{}

\begin{abstract}
We consider the problem of extending the classical S-lemma from
commutative case to noncommutative cases.  We show that a symmetric
quadratic homogeneous matrix-valued polynomial is positive semidefinite  if and only if its coefficient matrix is positive semidefinite. Then we extend the S-lemma to three kinds of noncommutative polynomials: noncommutative polynomials whose coefficients  are real numbers, matrix-valued noncommutative polynomials and hereditary polynomials. Some examples are given to demonstrate the relations between these new derived conditions.
\end{abstract}

\maketitle

\section{introduction}
	The classical S-lemma for commutative polynomials
	answers the question that when one quadratic inequality is a
	consequence of some other quadratic inequalities \cite{PITT2007}. Thus, it is a
	special form of Positivstellensatz from real algebraic geometry which
	characterizes polynomials that are positive (nonnegative) on a
	semialgebraic set \cite{Bochnak2003}. There are many important
	results of Positivstellensatz for  {\itshape noncommutative}
	cases. Helton proved a remarkable result that positive
	noncommutative polynomials are sums of squares~\cite{H2002}.
	Helton and McCullough presented a   noncommutative
	Postivestellensatz~\cite{HMS2004}.   Helton,  Klep, and McCullough
	gave a linear Positivestellensatz for characterizing the matricial
	linear matrix  inequality (LMI) domination
	problems~\cite{HKIMS2010}. Their result was    generalized by
	Zalar to solve the linear operator inequality (LOI) domination
	problems~\cite{Zar2017}.  When the domain is
	convex~\cite{HMS2004convex}, Helton,  Klep, and McCullough
	established a perfect noncommutative Nichtnegativstellensatz
	in~\cite{JIS2012}.  Furthermore, they   studied the matrix convex
	hulls of free semialgebraic set in~\cite{HKM2016}. Our goal in
	this paper is to investigate how to extend S-lemma to noncommutative cases.

   To state the main contributions of this paper, we need
   the following notations.
   The symbol $\mathbb{R}$ (resp. $\mathbb{N}$,
   $\mathbb{N}^+$) denotes the set of real (resp.  natural, positive
   natural) numbers.
For $n\in \mathbb{N}^{+}$, $\mathbb{R}^{n \times n}$ (resp.  $\mathbb{SR}^{n}$) stands
for set of $n\times n$ real matrices (resp. symmetric matrices).
 For $m, n\in \mathbb{N}^{+}$, the symbol $(\mathbb{R}^{n \times n})^{m}$ (resp.
$(\mathbb{SR}^{n})^m$ denotes the vector space consisting of
$m$-dimensional vectors of $n\times n$ real matrices (resp. symmetric
matrices).   The
   symbols  
   $\phi , \psi$ are used to represent the linear maps
   between finite dimensional  Euclidean spaces.

    The main results of the paper are stated below.
   We start with the simplest case where the coefficients of noncommutative polynomials are real numbers.

\begin{theorem}\label{thm3.1}
		Let
\[f=\sum_{i=1,j=1}^{m} a_{ij} x_i x_j, ~~g=\sum_{i=1,j=1}^{m} b_{ij} x_i x_j,\]
		be homogeneous quadratic symmetric noncommutative polynomials, where $a_{i,j}, b_{i,j}\in\mathbb{R}$ and $a_{ij}=a_{ji}$,
$b_{ij}=b_{ji}$ for all $i, j$.
		Suppose that there  is an $\hat{X} \in
		(\mathbb{SR}^{\hat{n}})^m$ for some $\hat{n} \in
		\mathbb{N}^{+}$,  such that $g(\hat{X}) \succ 0$.
		Then the following three statements are equivalent:
		\begin{enumerate}[\upshape (1)]
			\item For all $ X\in \mathbb{R}^m$, if $g(X) \ge 0$, then $f(X) \ge 0$.
			\item For all $ X\in (\mathbb{SR}^n)^m$,  $n \in
				\mathbb{N}^{+}$, if $g(X) \succeq 0$, then $f(X) \succeq 0$.
			\item There is a nonnegative real number $\lambda $ such
				that $f(X)-\lambda g(X) \succeq 0$ for all $ X\in
				(\mathbb{SR}^{n})^m$, $n \in \mathbb{N}^{+}$.
		\end{enumerate}

\end{theorem}

The  main part of the paper is devoted to extend  the S-lemma for  noncommutative polynomials with matrix coefficients, i.e. matrix-valued polynomials.
Let
$f(x)=\sum_{i=1,j=1}^{m} A_{ij} x_i x_j$  be a homogeneous quadratic symmetric matrix-valued polynomial,  where
$A_{ij}=A_{ji}^{T}, A_{ij}  \in \mathbb{R}^{q \times q}$ for all $i, j$. We show first in Theorem \ref{thm4.1} that
$f(X)=\sum_{i=1,j=1}^{m} A_{ij}\otimes X_i X_j \succeq 0 $ for all $X\in (\mathbb{SR}^{n })^{m}$ and $n\in\mathbb{N}^{+}$
if and only if its coefficient matrix $\mathcal A=(A_{ij}) \in \mathbb{SR}^{m q}$ is positive semidefinite.

{ For $n\in\mathbb{N}^+$, }let $\mathbbm{1}_n$ represent the identity map from $\mathbb{R}^{n
\times n}$ to $\mathbb{R}^{n \times n}$.
Inspired by Choi's characterization of a
completely positive map via a positive semidefinite Choi matrix
(Theorem \ref{thm2.2}), we generalize
the condition of existing a nonnegative number $\lambda$ such that
$f(X)-\lambda g(X) \geq 0$ for all $X\in\mathbb{R}^m$ to the existence
of a completely positive linear mapping $\phi:\mathbb{R}^{q \times
 	q} \rightarrow \mathbb{R}^{q \times q}$ such that $f(X)-(\phi \otimes
\mathbbm{1}_n) g(X) \succeq 0$ for all $ X\in (\mathbb{SR}^{n})^m,  n
\in \mathbb{N}^+$.

\begin{theorem} \label{thm1.1}
Let
\[f(x)=\sum_{i=1,j=1}^{m} A_{ij} x_i x_j, ~~
	g(x)=\sum_{i=1,j=1}^{m} B_{ij} x_i x_j,\]
 be homogeneous quadratic symmetric matrix-valued polynomials,  where
 $A_{ij}, B_{ij} \in \mathbb{R}^{q \times q}$ and $A_{ij}=A_{ji}^{T}$,
 $B_{ij}=B_{ji}^{T}$ for all $i, j$. Suppose that there is an $\hat{X}
 \in (\mathbb{SR}^{\hat{n}})^m$ for some $\hat{n} \in \mathbb{N}^{+}$,
 such that $g(\hat{X}) \succ 0$.  Then the following two statements
 are equivalent:
		\begin{enumerate}[\upshape (1)]
		\item For all $X\in (\mathbb{SR}^{n})^m$, $n>q$, if $({\rm Id}_{q} \otimes
			{P})g(X)({\rm Id}_{q} \otimes {P}) \succeq 0$, then $({\rm
			Id}_{q} \otimes {P})f(X)({\rm Id}_{q} \otimes
			{P}) \succeq 0$, where ${P}: \mathbb{R}^n \to
			\mathbb{R}^q$ is the projection to the last $q$
			coordinates.
			
			\item There is a nonzero completely positive linear
				mapping $\phi:\mathbb{R}^{q \times q} \rightarrow
				\mathbb{R}^{q \times q}$ such that $f(X)-(\phi \otimes
				\mathbbm{1}_n) g(X) \succeq 0$ for all $ X\in
				(\mathbb{SR}^{n})^m,  n \in \mathbb{N}^+$.
		\end{enumerate}
	\end{theorem}

The following theorem is for a special case of matrix-valued hereditary polynomials. 
 Its proof can be adjusted  from the proof of  Theorem \ref{thm1.1}.

\begin{theorem} \label{thm1.3}
		Let 
\[f(x)=\sum_{i=1,j=1}^{m}A_{ij} x_i x_j^T, ~~ g(x) =\sum_{i=1,j=1}^{m}B_{ij} x_i x_j^T, \]
be homogeneous matrix-valued hereditary polynomials, where $A_{ij}, B_{ij} \in
\mathbb{R}^{q \times q}$ and $A_{ij}=A_{ji}^{T}$, $B_{ij}=B_{ji}^{T}$
for all $i, j$. Suppose that there is an $\hat{X} \in
(\mathbb{R}^{\hat{n} \times \hat{n}})^m$ for some $\hat{n} \in
\mathbb{N}^{+}$, such that $g(\hat{X}) \succ 0$. Then the following two statements are equivalent:
		\begin{enumerate}[\upshape (1)]
			\item For all $ X\in (\mathbb{R}^{n \times n})^m$,  $n \in
				\mathbb{N}^{+}$,   if $g(X) \succeq 0$, then $f(X) \succeq 0$.
			\item There is a nonzero completely positive linear
				mapping $\phi:\mathbb{R}^{q \times q} \rightarrow
				\mathbb{R}^{q \times q}$, such that $f(X)-(\phi
				\otimes  \mathbbm{1}_n) g(X) \succeq 0$ for all $X\in
				(\mathbb{R}^{n \times n})^m$, $ n \in \mathbb{N}^+$.
		\end{enumerate}
	\end{theorem}

	\section{preliminaries}
	\subsection{Matrix-valued polynomials in symmetric entries}
In this paper, we deal with matrix-valued noncommutative polynomials. Different from the commutative polynomials, the variables and coefficients are all matrices.
The polynomial $p$  we considered in this paper has the following form:
	\[ p=\sum_{\omega \in \mathcal{W}_m} p_{\omega} \omega,\]
	where $p_{\omega} \in$ { $\mathbb{R}^{q \times q}$, $q\in \mathbb{N}^{+}$} and $\mathcal{W}_{m}$ is a set of words generated by the entries of $x=[x_1,x_2, \ldots,x_m]^{T}$, and 
	\[p^{T}=\sum_{\omega \in \mathcal{W}_m} p_{\omega}^{T} \omega ^{T}.\]
	If $p=p^{T}$, we say $p$ is symmetric. When we evaluate a polynomial $p$ at $X \in (\mathbb{SR}^{n})^{m}$, we  define the empty word as $\rm{Id}_n$,  where  ${\rm
Id}_{n}$ denote the identity matrix in $\mathbb{R}^{n \times n}$ for $n\in \mathbb{N}^{+}$.

	{
For symmetric quadratic homogeneous matrix-valued polynomials
\[
	f(x)=\sum_{i=1,j=1}^{m} A_{ij} x_i x_j\quad\text{and}\quad
g(x)=\sum_{i=1,j=1}^{m} B_{ij} x_i x_j,
\]
where $A_{ij}=A_{ji}^{T},B_{ij}=B_{ji}^{T},A_{ij} \ B_{ij} \in
\mathbb{R}^{q \times q}$, the evaluations of $f$ and $g$ at $X \in
(\mathbb{SR}^{n})^{m}$ are
\[
f(X)=\sum_{i=1,j=1}^{m} A_{ij}\otimes X_i X_j\quad\text{and}\quad
g(X)=\sum_{i=1,j=1}^{m} B_{ij}\otimes X_i X_j.
\]
}

	If we restrict the coefficients being real numbers, i.e.,  $q=1$, then we have  noncommutative polynomial
\[
p=\sum_{\omega \in \mathcal{W}_m} p_{\omega} \omega, ~ p_{\omega} \in \mathbb{R}. \]
	\subsection{The classical S-lemma}
	When $f$ and $g$ are homogeneous quadratic  polynomials, there are many different approaches for proving the S-lemma in commutative case. In \cite{VA1971}, Yakubovich used the convexity result  in \cite{Dines1941} to prove the S-lemma. A modern proof can be found in the book
by Ben-Tal and  Nemirovski \cite{AA2001}. An elementary proof of the S-lemma could be derived  based on a lemma given by  Yuan~\cite{Yuan1990}. See
excellent survey on S-lemma by P\'olik and Terlaky~\cite{PITT2007}. We
introduce  below  one of them which is suitable  for being extended to
the noncommutative cases.
	\begin{theorem}\label{s-lemmacom}
		Given $f,g:\mathbb{R}^{m}\rightarrow \mathbb{R}$ are homogeneous quadratic polynomials, and suppose there is an $\hat{X}\in \mathbb{R}^{m}$ such that $g(\hat{X})>0$. Then the following two statements are equivalent.
		\begin{enumerate}[\upshape (1)]
			\item For all $ X \in \mathbb{R}^{m}$, if $g(X) \geq 0$, then $f(X) \geq 0$.
			\item There is a nonnegative real number  $\lambda$ such that $f(X)-\lambda g(X) \geq 0$ for all $X \in \mathbb{R}^{m}$.
		\end{enumerate}
	\end{theorem}
	\begin{proof}
		Let $f(x)$, $g(x)$ be homogeneous quadratic polynomials. There are symmetric matrices  $A$, $B \in \mathbb{R}^{m \times m}$ such that
		\[ f(x)=x^{T} A x, ~~  g(x)=x^{T} B x.\]
	It is well known that  $h(X)=X^{T} H X \geq 0$ for all $X \in \mathbb{R}^{m}$ if and only if $H\succeq 0$.

The implication (2)$\Longrightarrow$(1) is obvious.  Now, assume that the condition  (2) is false, we show the condition (1) is false too. Consider two convex closed sets \[C=\{M\succeq 0 ~|~ M \in \mathbb{R}^{m \times m}\},\]
and
\[D=\{A-\lambda B ~|~ \lambda \geq 0\}.\]
As the condition  (2) is false, $C\cap D=\emptyset$, i.e., there is
  no nonnegative real number   $\lambda$ such that $A-\lambda B
  \succeq 0$.  Since  there is an $\hat{X}$ such that $g(\hat{X})>0$,
  we have $g(\hat{X})={\hat{X}}^{T} B \hat{X} >0$, which  means $B$
  must have a positive eigenvalue. Therefore, there must exist  a
  large enough positive real number  $\lambda_0$ such that
  $A-\lambda_0 B$ has a negative eigenvalue.
   Therefore, for  $\lambda>\lambda_0$,  the distance between $A-\lambda B$ and $C$ will get  larger when  $\lambda \rightarrow + \infty $.
    The topology and distance we used here are the general topology and distance of finite dimensional real Euclidean space.

	It is clear that
    \begin{align*}
    &\inf\{\|M_1-M_2\|~|~ M_1\in C,\ M_2 \in D\}\\
    =&\inf\{\|M_1-M_2\|~|~M_1\in C,\ M_2 \in \{A-\lambda B ~|~\lambda \leq  \lambda_0\}\}=d>0.
    \end{align*}
     By the separation theorem~\cite[Theorem 11.4]{RTR1970}, there exists an $S\in \mathbb{R}^{m \times m}$, $S\neq 0$ such that
		\[	\langle S,M_1 \rangle \geq a > \langle S,M_2 \rangle, ~{\rm for ~all}~  \quad ~M_1 \in C,M_2 \in D. \]
	As $C$ is a positive semidefinite cone,  $S\succeq0$ and $a=0$.

Since $\langle S, A-\lambda B \rangle < 0$ for all $\lambda \geq0$.
Let us  assume $\lambda =0$, then  we have $\langle S,A \rangle <0$.
Let $\lambda \rightarrow +\infty$, we have  $\langle S,B \rangle \geq
0$. Since $S$ is positive semidefinite, if   $\langle S,A \rangle <0$
and  $\langle S,B \rangle \geq 0$, according to Corollary 6.1.4 in
\cite{ML2012}, there exists an $X\in \mathbb{R}^{m}$ such that
$X^{T}AX<0$ and $X^{T}BX \geq 0$. Hence we have found an $X \in
\mathbb{R}^{m}$ such that $g(X)\geq 0$ and $f(X)<0$, which contradicts the condition (1).
	\end{proof}

	\subsection{Completely positive linear map} 
 A real number can be seen as a linear map form $\mathbb{R}$ to $\mathbb{R}$, and if the number is positive, the linear map translates a positive real number to a positive real number. Similarly,  we can define  positive linear maps and completely positive linear maps between real vector spaces of higher dimensions.

	A linear map $\phi:\mathbb{R}^{s \times s}\rightarrow \mathbb{R}^{t \times t}$, where  $s,t \in \mathbb{N}^+$ can be represented by a matrix in $\mathbb{R}^{(s\times t)\times (s \times t)}$
	\begin{align}\label{jphi}
\mathbf{J}(\phi)=\sum_{a,b=1}^{s} \phi(E_{ab}) \otimes E_{ab}=
	\begin{pmatrix}
		J_{11} &\cdots &J_{1t}\\
		\vdots& \ddots& \vdots\\
		J_{t1} &\cdots &J_{tt}
	\end{pmatrix},
\end{align}
	where $J_{ij} \in \mathbb{R}^{s \times s}$ and $E_{ab} \in \mathbb{R}^{s \times s}$ are   matrices  whose $(a,b)$-th entry is $1$ and all others are
	$0$. The matrix $\mathbf{J}(\phi )$ is called the Choi matrix of $\phi$ ~\cite{MDC1975}. It is easy to verify that for any $M\in\mathbb{R}^{s
		\times s}$,
\[
\phi(M)=
	\begin{pmatrix}
		\langle J_{11},M \rangle  &\cdots &\langle J_{1t},M \rangle \\
		\vdots& \ddots& \vdots\\
		\langle J_{t1},M \rangle &\cdots &\langle J_{tt},M \rangle
	\end{pmatrix}.
\]
	We say that the linear map $\phi$ is positive, if for every positive semidefinite matrix  $ M \in \mathbb{R}^{s \times s}, M \succeq 0$, its image under the map $\phi$ is also positive semidefinite, i.e., $\phi(M) \succeq 0$.
 Recall that $\mathbbm{1}_n$ represents the identity map from
 $\mathbb{R}^{n \times n}$ to $\mathbb{R}^{n \times n}$. We say $\phi$
 is completely positive, if for all $ n \in \mathbb{N}^+$, the linear
 map  $\phi \otimes  \mathbbm{1}_n$ is a positive linear map  from
 $\mathbb{R}^{(sn) \times (sn)}$ to $\mathbb{R}^{(sn) \times (sn)}$.

	\begin{theorem}\label{thm2.2}
		~\cite{MDC1975}
		The linear map $\phi : \mathbb{R}^{s \times s} \rightarrow \mathbb{R}^{t \times t}$ where $s,t \in \mathbb{N}^+$ is completely positive, if and only if the Choi matrix $\mathbf{J}(\phi) \succeq 0$.
	\end{theorem}
There is a one-to-one  correspondence  between the set of all
completely positive maps   from $\mathbb{R}^{s \times s}$ to
$\mathbb{R}^{t \times t}$ and the set of  positive semidefinite
matrices in $\mathbb{R}^{(st)\times (st)}$.
	\section{S-lemma of noncommutative polynomials}
	In this section, we prove the S-lemma for
	noncommutative polynomials (Theorem \ref{thm3.1}). Suppose that we
	are given polynomials
\[
	f(x)=\sum_{i=1,j=1}^{m} a_{ij} x_i x_j\quad\text{and}\quad
	g(x)=\sum_{i=1,j=1}^{m} b_{ij} x_i x_j,
\]
where $a_{i,j}, b_{i,j}\in\mathbb{R}$ and $a_{ij}=a_{ji}$,
$b_{ij}=b_{ji}$ for all $i, j$. Define
\[A=\begin{pmatrix}
	a_{11}&\cdots&a_{1m}\\
	\vdots&\ddots&\vdots\\
	a_{m1}&\cdots&a_{mm}\\
\end{pmatrix}
\quad\text{and}\quad
B=\begin{pmatrix}
	b_{11}&\cdots&b_{1m}\\
	\vdots&\ddots&\vdots\\
	b_{m1}&\cdots&b_{mm}\\
\end{pmatrix}.
\]

\noindent{\itshape Proof of Theorem \ref{thm3.1}.}
		(3)$\Longrightarrow$(2)$\Longrightarrow$(1): The implications are obvious.
		
		(1)$\Longrightarrow$(3): Assume that for all $X \in
		\mathbb{R}^m$, $g(X)=X^{T} B X \le 0$. Then we know $B\preceq
		0$ and hence
		\[g(X)=X^{T}(B\otimes {\rm Id}_{n}) X \preceq 0,\]
		for all $X \in (\mathbb{SR}^{n})^m$, $n\in\mathbb{N}^+$,
which contradicts the condition that there is an $\hat{X} \in
(\mathbb{SR}^{\hat{n}})^m$ for some $\hat{n} \in \mathbb{N}^{+}$, such that
$g(\hat{X}) \succ 0$.  Hence, there always exists  an $\tilde{X}\in
\mathbb{R}^{m}$ such that $g(\tilde{X})>0$. According to Theorem
\ref{s-lemmacom},  we can derive that  there exists a positive
 real number $\lambda $ such that $f(X)-\lambda g(X) \geq 0$ for all $ X\in \mathbb{R}^m$, especially $A-\lambda B \succeq 0$. Then we know
 \[f(X)-\lambda g(X)=X^{T}((A-\lambda B)\otimes {\rm Id}_{n}) X \succeq 0,\]
 for all $ X\in (\mathbb{SR}^{n})^m$, $n \in \mathbb{N}^{+}$.
 \qed

	\section{Positivity of symmetric quadratic homogeneous matrix-valued polynomials}
For a commutative polynomial $h(X)=X^{T} H X$ with
$H\in\mathbb{SR}^{m}$, we know $h(X)\geq 0$ for all $X \in
\mathbb{R}^{m}$ if and only if $H\succeq 0$. 	
It is very interesting to see that this property can be extended to noncommutative polynomials.
	\begin{theorem} \label{thm4.1}
		Let $f(x)=\sum_{i=1,j=1}^{m} A_{ij} x_i x_j$ be a symmetric
		quadratic homogeneous matrix-valued polynomial, where the
		matrices  $A_{ij}=A_{ji}^{T} \in \mathbb{R}^{q\times q}$ for
		all $i, j$. Define the coefficient matrix
		\[ \mathcal{A}=
		\begin{pmatrix}
			A_{11} & \cdots & A_{1m}\\
			\vdots & \ddots & \vdots \\
			A_{m1} & \cdots & A_{mm}
		\end{pmatrix}. \]
	Then $f(X) $ is positive semidefinite for all $X\in
	(\mathbb{SR}^{n })^{m}$, $n\in\mathbb{N}^{+}$,
 if and only if $\mathcal{A}$ is positive semidefinite.

	\end{theorem}
	
	\begin{proof}
		Let us  rearrange the matrix $\mathcal{A}$ to define a matrix in $\mathbb{R}^{q \times q} \otimes \mathbb{R}^{m \times m}$.
\begin{align}\label{A'}	
\mathcal{A}^{'}=
    \begin{pmatrix}
	\mathcal{A}^{'}_{11} &\cdots &\mathcal{A}^{'}_{1q} \\
	\vdots& \ddots& \vdots\\
	\mathcal{A}^{'}_{q1} &\cdots &\mathcal{A}^{'}_{qq}
    \end{pmatrix}=
\end{align}
\begin{align*}
		\begin{pmatrix}
		\langle A_{11},E_{11} \rangle & \cdots &\langle A_{1m},E_{11} \rangle& & \langle A_{11},E_{1q} \rangle & \cdots & \langle A_{1m},E_{1q} \rangle\\
		\vdots & \ddots &\vdots &\cdots & \vdots & \ddots & \vdots\\
		\langle A_{m1},E_{11} \rangle & \cdots &\langle A_{mm},E_{11} \rangle & & \langle A_{m1},E_{1q} \rangle & \cdots &\langle A_{mm},E_{1q} \rangle\\
		& \vdots & &\ddots &  & \vdots & \\
		\langle A_{11},E_{q1} \rangle & \cdots &\langle A_{1m},E_{q1} \rangle & & \langle A_{11},E_{qq} \rangle & \cdots & \langle A_{1m},E_{qq} \rangle\\
		\vdots & \ddots &\vdots &\cdots & \vdots & \ddots & \vdots\\
		\langle A_{m1},E_{q1} \rangle & \cdots &\langle A_{mm},E_{q1} \rangle & &\langle A_{m1},E_{qq} \rangle & \cdots & \langle A_{mm},E_{qq} \rangle\\
	\end{pmatrix}.
\end{align*}
Using the matrix $\mathcal{A}^{'}$ as the Choi matrix, define a linear map
\begin{equation}\label{phif}
 \begin{aligned}
 \psi _{f} :\mathbb{R}^{m \times m} &\rightarrow \mathbb{R}^{q \times q} \nonumber \\
		    M&\mapsto\begin{pmatrix}
			\langle \mathcal{A}^{'}_{11},M \rangle  &\cdots &\langle \mathcal{A}^{'}_{1q},M \rangle \\
			\vdots& \ddots& \vdots\\
			\langle \mathcal{A}^{'}_{q1},M \rangle &\cdots &\langle \mathcal{A}^{'}_{qq},M \rangle
		\end{pmatrix}.
\end{aligned}
\end{equation}
It is essential  to notice that
\[f(X)=\psi_{f} \otimes  \mathbbm{1}_n\begin{pmatrix}
			X_{1}X_{1} & \cdots & X_{1}X_{m}\\
			\vdots & \ddots &\vdots \\
			X_{m}X_{1} & \cdots &X_{m}X_{m}
		\end{pmatrix}\]
for all  $X\in (\mathbb{SR}^{n })^{m}$, $n\in \mathbb{N}^+$.

	Let $\{\alpha_1, \alpha_2, \ldots, \alpha_q\}$ be the standard orthogonal basis of $\mathbb{R}^q$, $\{\beta_1, \beta_2, \ldots, \beta_m\}$ be the standard orthogonal basis of $\mathbb{R}^m$, and
     \[u=\sum_{j=1}^{q} \sum_{i=1}^{m} (\alpha_j \otimes
	\beta_i)(\beta_i \otimes \alpha_j)^{T}. \]

	Now let us  assume that $\mathcal{A}$ is positive semidefinite.
	As the matrix $\mathcal{A}^{'}$ is obtained after applying unitary
	transformation by $u$ to the matrix $\mathcal{A}$,
	$\mathcal{A}^{'}$  is still positive semidefinite.
 According to Theorem \ref{thm2.2}, the linear map $\psi_{f}$ is completely positive.
  Hence, $\psi_{f} \otimes  \mathbbm{1}_n$ is a positive linear map for all  $ n \in \mathbb{N}^{+}$. Since
  \[\begin{pmatrix}
			X_{1}X_{1} & \cdots & X_{1}X_{m}\\
			\vdots & \ddots &\vdots \\
			X_{m}X_{1} & \cdots &X_{m}X_{m}
		\end{pmatrix}=\begin{pmatrix} X_1\\
\vdots\\
 X_m \end{pmatrix}\cdot (X_1, \ldots, X_m)\]
 is positive semidefinite,  we know that $f(X)$ is positive semidefinite for all  $X\in (\mathbb{SR}^{n })^{m}$, $n\in \mathbb{N}^+$.

On the other hand,  we  define $f'(X)=\sum_{i=1,j=1}^{m}  X_i X_j\otimes A_{ij}$.
 It is obvious that for any $X\in (\mathbb{SR}^{n })^{m}$, $n\in
 \mathbb{N}^+$,
 \[
 f(X)\succeq 0 \Leftrightarrow f'(X) \succeq 0.
 \]

  Let $X^0:=(X_1^0,\ldots,X_m^0) \in (\mathbb{SR}^{(m+1)})^m$,
  where each $X^0_i$ is the matrix whose
  $(1,i+1)$-th entry and $(i+1,1)$-th entry are $1$ and all others are
  $0$, i.e.,
		\[
			X^0_i:=	\bordermatrix{ & & &(i+1)\text{-th}& & \cr
	& 0&\cdots&1&\cdots&0\cr
	& \vdots &\ddots &\vdots &\ddots &\vdots\cr
	(i+1)\text{-th}& 1 & \cdots&0 &\cdots &0\cr
	& \vdots &\ddots &\vdots &\ddots &\vdots\cr
	& 0 & \cdots& 0&\cdots &0}.
		\]
		It is easy to check that
		\[f'(X^0)=	\begin{pmatrix}
			\sum_{i=1}^{m} A_{ii}& & & \\
			&A_{11}&\cdots&A_{1m}\\
			&\vdots&\ddots&\vdots\\
			&A_{m1}&\cdots&A_{mm}\\
		\end{pmatrix}.\]
By assumption $f'(X^0)\succeq 0$, and hence we have $\mathcal{A}\succeq 0$.
	\end{proof}
From this theorem, using the  spectral decomposition of $\mathcal{A}$,
we can factorize  the symmetric quadratic homogeneous matrix-valued
polynomials $f(x)$ which is positive semidefinite on
$(\mathbb{SR}^{n})^m$ for all $ n \in \mathbb{N}^{+}$ into the product
of  a linear homogeneous matrix-valued polynomial $U(x)$ and its
transpose $U(x)^T$
	\[f(x)=U(x)U(x)^{T}.\]
	 The dimension of  coefficients of the polynomial  $U$ is at most $(qm)\times (qm)$.

	 Let $h(x)$ be  a  matrix-valued polynomial having $m$ variables,
	 its degree is at most $2l$ and coefficients are matrices
	 belonging to
	 $\mathbb{R}^{q\times q}$. If $h(X)$ is positive semidefinite for
	 all $X\in (\mathbb{SR}^{n})^m$, $ n \in \mathbb{N}^{+}$,
	 McCullough has already shown   there exists a matrix-valued
	 polynomial $U(x)$, whose coefficients belong to
	 $\mathbb{R}^{(q\sum_{j=0}^{l}m^j)\times (q\sum_{j=0}^{l}m^j)}$,
	 such that $h(x)=U(x)^{T}U(x)$ \cite[Theorem  0.2]{MC2001}.
	 However, as the proofs use Arveson's extension theorem
	 \cite{Arverson69,paulsen_2003,Paulsen87}, it is unclear how to
	 construct the factorization.

	\section{S-lemma of matrix-valued polynomials}
	
Suppose that we are given polynomials
\[
	f(x)=\sum_{i=1,j=1}^{m} A_{ij} x_i x_j\quad\text{and}\quad
	g(x)=\sum_{i=1,j=1}^{m} B_{ij} x_i x_j,
\]
where $A_{ij}, B_{ij} \in \mathbb{R}^{q \times q}$ and
$A_{ij}=A_{ji}^{T}$, $B_{ij}=B_{ji}^{T}$ for all $i, j$.
In this section, we prove the S-lemma of matrix-valued polynomials
(Theorem \ref{thm1.1}).
Define
 \[\mathcal{A}=\begin{pmatrix}
		A_{11}&\cdots&A_{1m}\\
		\vdots&\ddots&\vdots\\
		A_{m1}&\cdots&A_{mm}\\
	\end{pmatrix}\quad\text{and}\quad
\mathcal{B}=\begin{pmatrix}
		B_{11}&\cdots&B_{1m}\\
		\vdots&\ddots&\vdots\\
		B_{m1}&\cdots&B_{mm}\\
	\end{pmatrix}.
\]

\noindent {\itshape Proof of Theorem \ref{thm1.1}.}
Assume that the condition (2) is satisfied.
Let  ${P}: \mathbb{R}^n \to
			\mathbb{R}^q$ be  the projection to the last $q$
			coordinates (actually $P$ could be any orthogonal projection matrix in $\mathbb{R}^{n \times n}$, or any matrix $Q\in \mathbb{R}^{n\times
				\ell}$, $\ell, n \in \mathbb{N}^{+}$,
see  Corollary \ref{cor5.1}).  For all $ X \in (\mathbb{SR}^{n})^{m}$, $ n,m \in \mathbb{N}^{+}$,  we have
		\begin{align*}
&\sum_{i=1,j=1}^{m} A_{ij}\otimes X_i X_j - \sum_{i=1,j=1}^{m} \phi  (B_{ij})\otimes X_i X_j \succeq 0 \\
 \Longrightarrow &\sum_{i=1,j=1}^{m} A_{ij}\otimes X_i X_j - \sum_{i=1,j=1}^{m} (\phi \otimes  \mathbbm{1}_n)  (B_{ij}\otimes X_i X_j) \succeq 0\\
\Longrightarrow & ({\rm Id}_{q}\otimes P) \left(\sum_{i=1,j=1}^{m}
A_{ij}\otimes X_i X_j\right) ({\rm Id}_{q}\otimes P)\\
	& - ({\rm Id}_{q}\otimes P)\left(\sum_{i=1,j=1}^{m} (\phi \otimes
	 \mathbbm{1}_n)  (B_{ij}\otimes X_i X_j)\right) ({\rm Id}_{q}\otimes
	P)\succeq 0\\
		 \Longrightarrow & ({\rm Id}_{q}\otimes P)
		 \left(\sum_{i=1,j=1}^{m} A_{ij}\otimes X_i X_j\right) ({\rm
		 Id}_{q}\otimes P)\\
		&-(\phi \otimes  \mathbbm{1}_n) ({\rm Id}_{q}\otimes
		P)\left(\sum_{i=1,j=1}^{m}B_{ij}\otimes X_i X_j\right)
		({\rm Id}_{q}\otimes P)\succeq 0.\\
  \Longrightarrow & ({\rm Id}_{q}\otimes P)
		 f(X)  ({\rm
		 Id}_{q}\otimes P) \succeq
		(\phi \otimes  \mathbbm{1}_n) ({\rm Id}_{q}\otimes
		P) g(X)
		({\rm Id}_{q}\otimes P).
\end{align*}
As $\phi$ is a completely positive linear map,
if $({\rm Id}_{q} \otimes {P})g(X)({\rm Id}_{q} \otimes {P}) \succeq 0$,
then  $(\phi \otimes  \mathbbm{1}_n) ({\rm Id}_{q}\otimes P) g(X) ({\rm
Id}_{q}\otimes P) \succeq 0$ and hence
$({\rm Id}_{q} \otimes {P})f(X)({\rm Id}_{q} \otimes {P}) \succeq 0$.
 The condition (1) is established.

 Now we assume that the condition  (2) is false, our aim is to show that the condition (1) is also false.  For a fixed linear map $\phi: \mathbb{R}^{q \times q} \rightarrow \mathbb{R}^{q \times q}$, let
 \begin{align}
 	\phi g(x)=\sum_{i=1,j=1}^{m}\phi (B_{ij})x_ix_j.
 \end{align}
 Consider the set
 \[\{f(x)-\phi g(x)~|~\phi:\mathbb{R}^{q \times
 	q} \rightarrow \mathbb{R}^{q \times q} ~{\rm  {is ~a ~completely~
	positive ~linear ~map}}\}.\]
 $f(x)-\phi g(x)$ is a homogeneous quadratic polynomial and its coefficient matrix has the following form
		\[\begin{pmatrix}
			A_{11}-\phi (B_{11})& \cdots & A_{1m}-\phi (B_{1m})\\
			\vdots & \ddots  & \vdots \\
			A_{m1}-\phi (B_{m1}) & \cdots & A_{mm}-\phi (B_{mm})
		\end{pmatrix}
		=\mathcal{A}-( \mathbbm{1}_m \otimes \phi) \mathcal{B}.\]
The set $\mathcal{D}$ which contains all such matrices is a closed convex cone in $\mathbb{SR}^{mq}$. Let $\mathcal{C}$ denote the positive semidefinite cone in $\mathbb{SR}^{mq}$.
Since the condition (2) is false, according to  Theorem \ref{thm4.1}, the coefficient matrix of $f(x)-\phi g(x)$  can not be positive semidefinite. Hence, we have $\mathcal{C}\cap \mathcal{D}=\emptyset$.

  Let us define
  \begin{align*}
 K=\{\mathbf{J}(\phi) ~|~ \phi:\mathbb{R}^{q \times
 	q} \rightarrow \mathbb{R}^{q \times q}  ~{\rm is ~ a ~ completely
	~ positive ~linear ~map}, ||\mathbf{J}(\phi)||=1\},
  \end{align*}
 where $\mathbf{J}(\phi)$ is defined by (\ref{jphi}).  The set
  $K$ is compact. For any completely positive linear map $\phi:\mathbb{R}^{q \times
  	q} \rightarrow \mathbb{R}^{q \times q}$ with $||\mathbf{J}(\phi)||=1$, define
  \begin{align*}
   \mathcal{D}_{\mathbf{J}(\phi)}=\{\mathcal{A}-\lambda( \mathbbm{1}_m \otimes \phi) \mathcal{B}~|~\lambda \geq 0 \},
    \end{align*}
and
		\[k(\mathbf{J}(\phi))=\inf\{\|\mathcal{M}_1-\mathcal{M}_2\|~|~\mathcal{M}_1 \in \mathcal{C},\ \mathcal{M}_2 \in \mathcal{D}_{\mathbf{J}(\phi)}\}.\]
		Then, $k(\mathbf{J}(\phi))$ can be seen as a continuous function on $K$.
		Since $K$ is compact, there is a completely positive linear
		map $\phi ^{0}$ and ${\mathbf{J}(\phi ^{0})} \in K$, such that
		$k(\mathbf{J}(\phi ^{0}))=\min_{\mathbf{J}(\phi)\in K} k(\mathbf{J}(\phi))$.
		
	For the completely positive linear map $\phi^{0}$, we have
	\begin{equation}\label{precond}
	\begin{aligned}	
		({\phi}^{0}\otimes  \mathbbm{1}_{\hat{n}})
		g(\hat{X})&=\sum_{i=1,j=1}^{m}{\phi}^{0}(B_{ij})\otimes
		{\hat{X}}_i {\hat{X}}_j\succeq 0,\\
		({\phi}^{0}\otimes  \mathbbm{1}_{\hat{n}}) g(\hat{X})&=\sum_{i=1,j=1}^{m}{\phi}^{0}(B_{ij})\otimes {\hat{X}}_i {\hat{X}}_j\neq 0.
    \end{aligned}
\end{equation}
	Now we show that $( \mathbbm{1}_m \otimes {\phi}^{0}) \mathcal{B}$ has a positive eigenvalue. If not, we have
	\[( \mathbbm{1}_m \otimes {\phi}^{0}) \mathcal{B} \preceq 0.\]
	Then $({\phi}^{0}\otimes  \mathbbm{1}_n)g (X) \preceq 0$ for all $ X
	\in  (\mathbb{SR}^{n})^{m}$, which contradicts (\ref{precond}).	
	The condition that $( \mathbbm{1}_m \otimes \phi ^{0}) \mathcal{B}$ has a positive eigenvalue ensures that
\[k(\mathbf{J}(\phi ^{0}))=d >0.\]
  Therefore, we have
  \[\inf\{\Vert M_1-M_2\Vert | M_1\in \mathcal{C},\ M_2 \in \mathcal{D}\}=d>0.\]
   By the separation theorem~\cite[Theorem 11.4]{RTR1970}, there is a   matrix  $M^s\in \mathbb{R}^{(mq)\times (mq)}$, such that
		\[\langle M_1,M^s \rangle \geq a_0 >\langle M_2,M^s \rangle,
		\quad \forall M_1\in \mathcal{C},\ M_2 \in \mathcal{D}. \]
		It is clear that $M^s \succeq 0$ and $a_0=0$. Then we have
\begin{align}\label{sepeq}		
\langle \mathcal{A},M^s \rangle<0 \ ~{\rm and}~ \ \langle ( \mathbbm{1}_m \otimes \phi) \mathcal{B},M^s\rangle \geq 0,
\end{align}
		for every completely positive linear map $\phi: \mathbb{R}^{q \times q} \rightarrow \mathbb{R}^{q \times q}$.

Let $\{e_1, e_2, \ldots, e_q\}$ be the standard orthogonal basis of $\mathbb{R}^q$, and $E=\sum_{i=1}^{q} e_i \otimes e_i$. The matrix $M^s$ can be written in  the following form
\begin{align}\label{Ms}
M^s=\begin{pmatrix}
			M^s_{11}& \cdots & M^s_{1m}\\
			\vdots & \ddots  & \vdots \\
			M^s_{m1}& \cdots & M^s_{mm}\\
		\end{pmatrix},
\end{align}
where each $M^s_{ij} \in \mathbb{R}^{q\times q}$.
The condition  (\ref{sepeq}) can be written in the following form:
\begin{align} \label{Aless0}
	\langle \mathcal{A},M^s \rangle=\sum_{i=1,j=1}^{m} \langle
	A_{ij},M^s_{ij}\rangle=E^T\left(\sum_{i=1,j=1}^{m} A_{ij}\otimes
	M^s_{ij}\right)E<0.
\end{align}
 Moreover, we have
 \begin{align*}
\langle ( \mathbbm{1}_m \otimes \phi) \mathcal{B},M^s \rangle
&=\sum_{i=1,j=1}^{m} \langle \phi (B_{ij}),M^s_{ij}\rangle \nonumber\\
&=\left\langle \sum_{i=1,j=1}^{m} \phi (B_{ij})\otimes
M^s_{ij},\sum_{a=1,b=1}^{q}E_{ab}\otimes E_{ab} \right\rangle \nonumber\\
&=\left\langle \sum_{i=1,j=1}^{m} B_{ij}\otimes
M^s_{ij},\sum_{a=1,b=1}^{q}\phi (E_{ab})\otimes E_{ab} \right\rangle \nonumber\\
& =\left\langle \sum_{i=1,j=1}^{m}  B_{ij}\otimes
M^s_{ij},\mathbf{J}(\phi ) \right\rangle \geq 0,
\end{align*}
{where $E_{ab} \in \mathbb{R}^{q \times q}$ are   matrices  whose
$(a,b)$-th entry is $1$ and all others are $0$. }
 According to Theorem \ref{thm2.2},  the set
\[
	\{\mathbf{J}(\phi )~|~\phi: \mathbb{R}^{q \times q} \rightarrow
	\mathbb{R}^{q \times q} \text{ is completely positive}\}
\]
is equivalent to the positive semidefinite cone in
$\mathbb{SR}^{q^2}$. We have
\begin{align} \label{Blarge0}
	\sum_{i=1,j=1}^{m}  B_{ij}\otimes M^s_{ij} \succeq 0.
\end{align}

 In order to show that the  condition (1) in Theorem \ref{thm1.1} is not
 satisfied,  we need  to translate the inequality conditions
 (\ref{Aless0}) and (\ref{Blarge0}) into the  evaluations of $f$ and
 $g$ at some matrix vector $X\in {(\mathbb{SR}^q)}^{m}$. Since the
 positive semidefinite matrix  $M^s=(M_{ij}^s) \in {(\mathbb{SR}^q)}^{m} $
may not belong to the set \[\mathcal{X}=\{Y Y^T~|~ Y \in {(\mathbb{SR}^q)}^{m}\},\]
which is a strict subset of the positive semidefinite cone
$\mathcal{C} \subset \mathbb{SR}^{mq}$.
 Hence,   we can not  ensure that there  always  exists an $X\in {(\mathbb{SR}^q)}^{m}$ such that
 \[
	f(X)=\sum_{i=1,j=1}^{m}  A_{ij}\otimes
	M^s_{ij}\quad\text{and}\quad   g(X)=\sum_{i=1,j=1}^{m}  B_{ij}\otimes M^s_{ij}.
\]
This is the main reason why we introduce a projection (\ref{projection})  to construct an evaluation point.

Since $M^s$ defined in  (\ref{Ms}) is a positive semidefinite matrix, it has the decomposition
\begin{align*}
M^s&=\sum_{k=1}^{r} v_k v_k^T,\  v_k\in \mathbb{R}^{mq},\
r=\text{\rm{rank}}(M^s),\\
v_k&=\begin{pmatrix}
			v_k^1\\
			\vdots \\
			v_k^{m} \\
		\end{pmatrix},\ v_k^l \in \mathbb{R}^q,\ 1 \leq l \leq m,\
		k=1,\ldots,r.
\end{align*}
	We define $X^M:=(X_1^M,\ldots,X_m^M) \in (\mathbb{R}^{(r+q)\times(r+q)})^m$,	where for each $i=1,\ldots,m,$
\begin{align}\label{constructX}
X^M_i=\begin{pmatrix}
			& & &(v_1^i)^T\\
			&0& &\vdots\\
			& & &(v_{r}^i)^T\\
			& & &\\
			v_1^i&\cdots&v_{r}^i&0
		\end{pmatrix},
\end{align}
 and the projection $P^M: \mathbb{R}^{(r+q)} \to \mathbb{R}^{q}$ to
 the last $q$ coordinates
\begin{align}\label{projection}
 P^M=\begin{pmatrix}
			0& \\
			& {\rm Id}_q\\
		\end{pmatrix}.
\end{align}
Then the condition (\ref{Blarge0}) can be used to show
\[
        \begin{aligned}
        ({\rm Id}_q \otimes P^M)g(X^M)({\rm Id}_q \otimes P^M)&=\sum_{i=1,j=1}^{m}  B_{ij}\otimes P^M {X^M_i}{X^M_j} P^M \\
        &=\sum_{i=1,j=1}^{m}  B_{ij}\otimes M^s_{ij} \succeq 0.
        \end{aligned}
	\]
On the other hand, the condition (\ref{Aless0}) can be used to show
\[
\begin{aligned}		
E^T({\rm Id}_q \otimes P^M)f(X^M)({\rm Id}_q \otimes P^M)E
		&=E^T\left(\sum_{i=1,j=1}^{m}  A_{ij}\otimes P^M {X^M_i}{X^M_j} P^M\right) E \\
&=E^T\left(\sum_{i=1,j=1}^{m}  A_{ij}\otimes M^s_{ij}\right) E<0.
\end{aligned}
\]
Therefore, we have
\[
({\rm Id}_q \otimes P^M)f(X^M)({\rm Id}_q \otimes P^M) \nsucceq 0.
\]
Hence, the condition (1) in Theorem \ref{thm1.1} is false.\qed

{
\begin{corollary}\label{cor5.1} Under the same assumption in Theorem
	\ref{thm1.1}, the statements $(1)$ and $(2)$ in Theorem \ref{thm1.1}
	are also equivalent to the following two conditions:
	\begin{enumerate}[\upshape (1)]\setcounter{enumi}{2}
			\item For all $ X\in (\mathbb{SR}^{n})^m$, orthogonal
				projection matrices $P\in \mathbb{R}^{n\times n}$, $n \in \mathbb{N}^{+}$, if $({\rm Id}_{q} \otimes
				P)g(X)({\rm Id}_{q} \otimes P) \succeq 0$, then $({\rm
				Id}_{q} \otimes P)f(X)({\rm Id}_{q} \otimes P) \succeq
			0$.
	\item  For all $ X\in (\mathbb{SR}^{n})^m$, $Q\in \mathbb{R}^{n\times
				\ell}$, $\ell, n \in \mathbb{N}^{+}$, if $({\rm Id}_{q} \otimes
				Q^T)g(X)({\rm Id}_{q} \otimes Q) \succeq 0$, then $({\rm
				Id}_{q} \otimes Q^T)f(X)({\rm Id}_{q} \otimes Q) \succeq
			0$.
			
		\end{enumerate}
\end{corollary}
\begin{proof}
	From the proof of $(2) \Rightarrow (1)$, we can see that $(2)
	\Rightarrow (4)$ also holds. The implications $(4) \Rightarrow (3)
	\Rightarrow (1)$ are obvious.
\end{proof}}

\begin{remark}
    Theorem \ref{thm1.1} is still true when the dimension $q_f$ of
	the coefficients of the  polynomial $f$  is smaller than the
	dimension $q_g$ of the coefficients of the polynomial $ g$.
	In fact, we can always add zeros to the coefficients of $f$ to
	make $q_f=q_g$. Consider the case when
	$q_f>q_g$. Suppose that $k$ is the
	smallest  positive integer satisfying  $q_f \leq k q_g$.
	Define a new polynomial $\tilde g=\oplus ^{k}g$.
	Then Theorem \ref{thm1.1} is still valid after replacing $g$ by $\tilde g$.
\end{remark}
	
	\section{Other variants of S-lemma in noncommutative cases}

	\subsection{Other variant conditions of S-lemma}
	Different from commutative polynomials, there are
	many ways to extend the S-lemma for (matrix-valued) polynomials with
	matrix evaluations. Comparing with the condition (1) in Theorem \ref{thm1.1}, we consider the following condition which is a more direct extension
	of the classical S-lemma:
	\begin{enumerate}[\upshape (1)]
		\item[$(1')$] For all $ X\in (\mathbb{SR}^{n })^m$, $n \in
			\mathbb{N}^{+}$, if $g(X) \succeq 0$, then $f(X) \succeq 0$.
	\end{enumerate}

\begin{remark}
It is straightforward to verify that  the  condition (2) in Theorem \ref{thm1.1} implies $(1')$.
	Therefore, under the assumption that  there is an $\hat{X} \in
(\mathbb{R}^{\hat{n} \times \hat{n}})^m$ for some $\hat{n} \in
\mathbb{N}^{+}$, such that $g(\hat{X}) \succ 0$,
  the
 condition (1) in Theorem \ref{thm1.1} implies the condition  $(1')$,
 but it is unknown
 if it is true the other way around.

 As illustrated by the following example, without the
 assumption of existing an $\hat{X}$ such that $g(\hat{X}) \succ 0$, $(1')$ can not imply the
 condition (1) in Theorem \ref{thm1.1}.
 \end{remark}

	\begin{example}\label{example6.1} We construct two matrix-valued polynomials
		\[f=\begin{pmatrix}
			x_1 x_2 +x_2 x_1 -x_2 x_2&0\\
			0&0
		\end{pmatrix} \oplus \begin{pmatrix}
			0&0\\
			0&0
		\end{pmatrix},\]
		\[g=\begin{pmatrix}
			x_1 x_1-x_2x_2&0\\
			0&x_1x_2+x_2x_1
		\end{pmatrix} \oplus \begin{pmatrix}
			0&x_1x_2-x_2x_1\\
			x_2x_1-x_1x_2&0
		\end{pmatrix}. \]
	\end{example}
	For any $ X\in (\mathbb{SR}^{n})^2$, if $g(X)\succeq 0$, we have $X_1X_2-X_2X_1=0$. So $X_1,X_2$ have the same eigenspaces. Let $X=(X_1, X_2)$, where
 \[X_1=\sum_{i=1}^{r}\lambda_i v_i v_i^T, ~X_2=\sum_{i=1}^{r}\mu_i v_i v_i^T.\]
 Assume that  $g(X)\succeq 0$,  then we have
 \[(\lambda_i)^2-(\mu_i)^2 \geq 0 ~{\rm and}~\lambda_i \mu_i \geq 0.\]
  It is easy to  check that  $f(X) \succeq 0$. Hence $f$ and $g$ satisfy the condition $(1')$.

{
  On the other hand,  let
	\[X^0_1=\begin{pmatrix}
		0&0&0\\
		0&0&\sqrt{2}\\
		0&\sqrt{2}&0
	\end{pmatrix} \oplus \begin{pmatrix}
	0&0&0\\
	0&0&0\\
	0&0&0
\end{pmatrix}
, \ X^0_2=\begin{pmatrix}
		0&0&1\\
		0&0&0\\
		1&0&0
	\end{pmatrix} \oplus \begin{pmatrix}
	0&0&0\\
	0&0&0\\
	0&0&0
\end{pmatrix}
,\] \[P=\begin{pmatrix}
		0&0\\
		0&0
	\end{pmatrix} \oplus {\rm Id_4}.\]
It is straightforward to verify that
\begin{align*}
	({\rm Id}_4\otimes P)g(X^0)({\rm Id}_4\otimes P)&=\begin{pmatrix}
		P(X^0_1 X^0_1-X^0_2X^0_2)P&0\\
		0&P(X^0_1X^0_2+X^0_2X^0_1)P
	\end{pmatrix} \\
  	&\oplus \begin{pmatrix}
	0&P(X^0_1X^0_2-X^0_2X^0_1)P\\
	P(X^0_2X^0_1-X^0_1X^0_2)P&0
\end{pmatrix}.
\end{align*}
 The top left corner matrix is positive semidefinite
\[P(X^0_1 X^0_1-X^0_2X^0_2)P=\begin{pmatrix}
0&0&0&0&0&0\\
0&0&0&0&0&0\\
0&0&1&0&0&0\\
0&0&0&0&0&0\\
0&0&0&0&0&0\\
0&0&0&0&0&0
\end{pmatrix}
 \succeq 0.
\]
The other submatrices  are all zero matrices
\[P(X^0_1X^0_2+X^0_2X^0_1)P=\pm P(X^0_1X^0_2-X^0_2X^0_1)P=0.\]
Therefore, we have
\[({\rm Id}_4\otimes P)g(X^0)({\rm Id}_4\otimes P) \succeq 0.\]
However, we have
\begin{align*}
({\rm Id}_4\otimes P)f(X^0)({\rm Id}_4\otimes P)
	&=\begin{pmatrix}
		P (X_1^0 X^0_2 +X^0_2 X^0_1 -X^0_2 X^0_2) P&0\\
		0&0
	\end{pmatrix} \oplus \begin{pmatrix}
		0&0\\
		0&0
	\end{pmatrix}.
\end{align*}
 The top left corner matrix is negative  semidefinite
 \[P (X_1^0 X^0_2 +X^0_2 X^0_1 -X^0_2 X^0_2) P=\begin{pmatrix}
0&0&0&0&0&0\\
0&0&0&0&0&0\\
0&0&-1&0&0&0\\
0&0&0&0&0&0\\
0&0&0&0&0&0\\
0&0&0&0&0&0
\end{pmatrix} \preceq 0.\]
 Therefore, we have
\[({\rm Id}_4\otimes P)f(X^0)({\rm Id}_4\otimes P) \preceq 0.\]
Therefore,  the condition (1) in Theorem \ref{thm1.1}  is false for
the given $f$ and $g$.\qed
}

{ With the assumption that there is an $\hat{X} \in
 (\mathbb{SR}^{\hat{n}})^m$ for some $\hat{n} \in \mathbb{N}^{+}$,
 such that $g(\hat{X}) \succ 0$, whether or not $(1')$ can imply the
 condition (1) in Theorem \ref{thm1.1} is an interesting problem and
 we wish to investigate it in future.
}

	Furthermore, one can also consider the following  condition:

	\begin{enumerate}[\upshape (1)]
		\item[$(1'')$] For all $ X\in (\mathbb{SR}^{n })^m$, $n \in \mathbb{N}^{+}$, given a vector  $v\in \mathbb{R}^{qn}$, if $v^{T}g(X)v \geq 0$ then $v^{T}f(X)v \geq 0$.
	\end{enumerate}

The following example shows that  the condition $(1'')$ is strictly
stronger than the condition (1) in Theorem \ref{thm1.1}  and the condition $(1')$.
\begin{example}\label{eg6.2} We are given   the   matrix-valued polynomials
		\[f=\begin{pmatrix}
			x_1x_1&0\\
			0&x_1x_1-x_2x_2
		\end{pmatrix},\]
and
		\[g=\begin{pmatrix}
			x_1x_1-x_2x_2&0\\
			0&x_1x_1
		\end{pmatrix}. \]
\end{example}
Let us define a   linear map $\phi_2$ from $\mathbb{R}^{2 \times 2}$ to $\mathbb{R}^{2 \times 2}$
\[\phi_2:~\begin{pmatrix}
	a&b\\
	c&d
\end{pmatrix}\rightarrow
\begin{pmatrix}
	d&0\\
	0&a
\end{pmatrix}.\]
It is easy to verify that $\phi_2$ is a completely positive linear map.
We have
\[f(X)-(\phi_2 g)(X)=0  ~{\rm  for ~all}~ X\in \mathbb{SR}^n,\ n\in \mathbb{N}^{+}.\]
The condition (2) in Theorem \ref{thm1.1}  is satisfied. Therefore, the condition (1) in  Theorem \ref{thm1.1}  and the condition $(1')$ are satisfied too.
However, let  \[X^0=[1,2]^T, ~v=[0,1]^T,\] we have
\[v^T g(X^0) v=1 >0, ~{\rm but}~ ~~v^T f(X^0)v=-3 <0.\]
Therefore the condition  $(1'')$ above is not satisfied.\qed

\subsection{Proof of Theorem \ref{thm1.3}}
	We assume that $f(x)$ and $g(x)$ are homogeneous matrix-valued polynomials with following form
	\[
f(x)=\sum_{i=1,j=1}^{m}A_{ij} x_i x_j^T, ~g(x)
=\sum_{i=1,j=1}^{m}B_{ij} x_i x_j^T,
\]
where $A_{i,j}, B_{i,j}\in\mathbb{R}^{q\times q}$ and
$A_{i,j}=A_{j,i}^T$, $B_{i,j}=B_{j,i}^T$ for all $i, j$.
Now we prove Theorem \ref{thm1.3} which implies that the condition
$(1)$ in Theorem \ref{thm1.1} can be simplified to $(1')$, i.e.,
we do not need projection for the matrix-valued hereditary
polynomials.	

Define that
\[\mathcal{A}=\begin{pmatrix}
	A_{11}&\cdots&A_{1m}\\
	\vdots&\ddots&\vdots\\
	A_{m1}&\cdots&A_{mm}\\
\end{pmatrix}\quad\text{and}\quad
\mathcal{B}=\begin{pmatrix}
	B_{11}&\cdots&B_{1m}\\
	\vdots&\ddots&\vdots\\
	B_{m1}&\cdots&B_{mm}\\
\end{pmatrix}.\]
Let $\mathcal{A}^{'}$ be defined as in (\ref{A'}) and $\psi_{f}:\mathbb{R}^{m
\times m} \rightarrow \mathbb{R}^{q \times q}$ be the linear map
defined by $\mathcal{A}^{'}$  (\ref{phif}). Then, it holds that
\[f(X)=\psi_{f} \otimes  \mathbbm{1}_n\begin{pmatrix}
	X_{1}{X_{1}}^{T} & \cdots & X_{1}{X_{m}}^{T}\\
	\vdots & \ddots &\vdots \\
	X_{m}{X_{1}}^{T} & \cdots &X_{m}{X_{m}}^{T}
\end{pmatrix}=\psi_{f} \otimes  \mathbbm{1}_n (XX^T),\]
for all  $X\in (\mathbb{R}^{n \times n})^{m}$, $n\in \mathbb{N}^+$.
Similarly,
let $\mathcal{B}^{'}$ be the rearrangement of $\mathcal{B}$ and
$\psi_{g}:\mathbb{R}^{m \times m} \rightarrow \mathbb{R}^{q \times q}$
be the linear map defined by $\mathcal{B}^{'}$ such that
  \[ g(X)=\psi_{g} \otimes  \mathbbm{1}_n (XX^T),\]
  for all  $X\in (\mathbb{R}^{n \times n})^{m}$, $n\in \mathbb{N}^+$.

\vskip 3pt
  \noindent {\itshape Proof of Theorem \ref{thm1.3}}\
	 The implication $(2) \Rightarrow (1)$ is obvious.
	
 Assume that the condition (2) in Theorem \ref{thm1.3} is false.  Similar
 to the discussion in  the proof of Theorem \ref{thm1.1},   we can find a
 separation matrix $M^s \succeq 0$ which satisfies the condition
 (\ref{sepeq})
and has the following decomposition:
\begin{align*}
M^s&=\sum_{k=1}^{r} v_k v_k^T,\  v_k\in \mathbb{R}^{mq},\
r=\text{\rm{rank}}(M^s),\\
v_k&=\begin{pmatrix}
			v_k^1\\
			\vdots \\
			v_k^{m} \\
		\end{pmatrix},\ v_k^l \in \mathbb{R}^q,\ 1 \leq l \leq m,\
		k=1,\ldots,r.
\end{align*}

  Since we do not require the variable $X^M_i$ to be symmetric, instead of constructing
 $X^M_i$ as in (\ref{constructX}), we let
	\[X^M_i=\begin{pmatrix}
		v^i_1&\cdots &v^i_{r}
	\end{pmatrix}.\]
 Letting $n=\max\{r,q\}$, we add zero rows or columns into $X_i^M \in
 \mathbb{R}^{q\times r}$ to make it a square matrix in
 $\mathbb{R}^{n\times n}$. Without loss of generality, we assume that
 $r>q$, and define   new matrices $\tilde X^M_i \in
 \mathbb{R}^{n\times n}$ for $i=1,\ldots,m$,
  \[\tilde X^M_i=\begin{pmatrix}
		v^i_1&\cdots &v^i_{r} \\
         0 & \cdots &0 \\
         \vdots &\ddots & \vdots \\
          0 & \cdots &0
	\end{pmatrix}. \]
 Let  $\tilde X^M:=(\tilde X_1^M,\ldots,\tilde X_m^M) \in (\mathbb{R}^{n \times n})^m$.
  We can  translate the inequality conditions  (\ref{Aless0}) and (\ref{Blarge0}) into the  evaluations of $f$ and $g$ at  $\tilde X^M \in {(\mathbb{R}^{n\times n})}^{m}$. In particular, we have
  \begin{align*}
    g(\tilde X^M)=\sum_{i=1,j=1}^{m}  B_{ij}\otimes \begin{pmatrix}
    	 M^s_{ij}&0\\
    	 0&0
    \end{pmatrix} \succeq 0.
\end{align*}
Let $\{e_1, e_2, \ldots, e_q\}$ be the standard orthogonal basis of $\mathbb{R}^q$,  $\{f_1, f_2, \ldots, f_n\}$ be the standard orthogonal basis of $\mathbb{R}^n$, and $E'=\sum_{i=1}^{q} e_i \otimes f_i$, we have
\begin{align*}
    E'^T f(\tilde X^M){E'}=E'^T\left(\sum_{i=1,j=1}^{m}  A_{ij}\otimes \begin{pmatrix}
    	M^s_{ij}&0\\
    	0&0
    \end{pmatrix} \right){E'} < 0.
     \end{align*}
  Therefore,   the condition (1) in Theorem \ref{thm1.3} is false.
  \qed

\subsection*{Some discussions}
{
In this paper, we show several variants of the S-lemma in
noncommutative cases for quadratic homogeneous  polynomials.  Unlike
the commutative case, the S-lemma for general quadratic nonhomogeneous
polynomials in noncommutative case is still unknown.

In the commutative case, it is straightforward to convert a nonhomogeneous
polynomial to a homogeneous one by introducing a new variable.
For example, let
\[
f(x)=\sum_{i=1,j=1}^{m} a_{ij}x_i x_j + \sum_{i=1}^{m} a_i x_i +a_0,
\]
where $a_{ij}=a_{ji}, a_i, a_0\in\mathbb{R}$ for all $i, j$.
By introducing a new variable $x_0$, the homogenization of $f(x)$
can be written in the following form:
\begin{align*}
\tilde f(x_0,x)=\sum_{i=1,j=1}^{m} a_{ij}x_i x_j + \sum_{i=1}^{m} a_i
x_i x_0 +a_0 x_0^2.
\end{align*}
Then we have
\[
	\tilde f(X_0, X)=X_0^2 f(X/X_0),  \quad\text{for all}\quad X \in
	\mathbb{R}^{m},\ X_0 \neq 0  \in \mathbb{R}.
\]
Using this fact, the proof of the classical S-lemma for commutative
nonhomogeneous polynomials can be reduced to homogeneous ones (see
\cite{VA1971}).

However, this process becomes more complicated in the noncommutative
cases. First of all, due to the noncommutativity of variables, the
homogenization of a noncommutative polynomial is not unique. For
example, consider a nonhomogeneous  quadratic matrix-valued
polynomial
\[
	f(x)=\sum_{i=1,j=1}^{m} A_{ij}x_i x_j +
	\sum_{i=1}^{m} A_i x_i +A_0,
\]
where $A_{ij}=A_{j,i}^T, A_i, A_0\in\mathbb{R}^{q\times q}$ for all
$i,j$.
By introducing a new variable $x_0$, we homogenize $f$ to
\begin{align*}
	h(x_0,x)=\sum_{i=1,j=1}^{m} A_{ij} x_i x_j+ \sum_{i=1}^{m} H_{i0} x_i x_0+ \sum_{i=1}^{m} H_{0i} x_0 x_i+A_0 x_0 x_0,
\end{align*}
where
\begin{align}\label{Hi0}
	~H_{i0}+H_{0i}=A_{i},  \quad \text{for all} \quad i=1,\cdots,m.
 \end{align}
There exist  different choices of $H_{i0}$ and $H_{0i}$
satisfying \eqref{Hi0} for $ 1 \leq i \leq m$. Therefore, the homogenization of a quadratic
nonhomogeneous noncommutative polynomial is not unique.

\begin{example}
For the quadratic nonhomogeneous noncommutative polynomial	
\[
		f(x)=\begin{pmatrix}
			x^2&x\\
			x& 1
		\end{pmatrix},
\]
we have two different choices of homogenization:
	\begin{align*}
	    h_1(x_0, x)=\begin{pmatrix}
	    	x^2&xx_0\\
	    	x_0x&x_0^2
		\end{pmatrix} \quad\text{and}\quad
        h_2(x_0,x)=\begin{pmatrix}
        	x^2&x_0x\\
        	xx_0&x_0^2
        \end{pmatrix}.
	\end{align*}
For all $X\in \mathbb{SR}^{n}$, $n\in \mathbb{N}^+$, it holds that
\begin{align*}
	f(X)=h_1({\rm Id_n},X)=h_2({\rm Id_n}, X).
\end{align*}
The   coefficient matrices of $h_1(X_0, X)$ and $h_2(X_0, X)$ satisfy the following conditions:
\[ \left( \begin {array}{cccc} 0&0&0&0\\ \noalign{\medskip}0&1&1&0
\\ \noalign{\medskip}0&1&1&0\\ \noalign{\medskip}0&0&0&0\end {array}\right)
\succeq 0, ~~~~~\, \left( \begin {array}{cccc} 0&0&0&1\\ \noalign{\medskip}0&1&0&0
\\ \noalign{\medskip}0&0&1&0\\ \noalign{\medskip}1&0&0&0\end {array}
 \right) \not\succeq  0.\]
By Theorem \ref{thm4.1}, we know that $h_1(X_0, X)$ is positive
semidefinite for all $X_0, X\in\mathbb{SR}^n$, $n\in\mathbb{N}^+$, while   $h_2(X_0, X)$ is not positive
semidefinite for all $X_0, X\in\mathbb{SR}^n$, $n\in\mathbb{N}^+$.

\qed
\end{example}
Nevertheless, given a positive semidefinite polynomial $f(x)$, there
always exists a choice of $H_{i0}$ and $H_{0i}$ satisfying
(\ref{Hi0}), such that
the homogenization $h(x)$ is positive semidefinite. In fact,
according to \cite[Theorem  0.2]{MC2001}, if a quadratic polynomial
$f(X)$ is  positive semidefinite for all $X\in (\mathbb{SR}^{n})^m$,
$n\in\mathbb{N}^+$,
then there exists a matrix-valued linear polynomial  $U(x)$, whose
coefficients belong to $\mathbb{R}^{(q (m+1)\times (q (m+1))}$, such
that $f(x)=U(x)^{T}U(x)$. Hence, we can let  $h(x_0, x)=\widetilde
U(x_0,x)^T \widetilde U(x_0,x)$
where $\widetilde U(x_0,x)$ is obtained by homogenizing  $U(x)$. It is
clear that $h(X)$  is  positive semidefinite  for all $X\in
(\mathbb{SR}^{n})^{m+1}$, $n\in\mathbb{N}^+$.  Thanks to
Theorem \ref{thm4.1}, one can find
such a homogenization $h(x)$ by solving a semidefinite
program with the positive semidefinite constraint of the coefficient matrice of
$h$ and the equality constraint \eqref{Hi0}.

However, unlike the commutative case proved in \cite{VA1971}, it is
unclear how to derive S-lemma for nonhomogeneous quadratic polynomials
from homogeneous ones. In particular, for a general nonhomogeneous
quadratic polynomial $f$ and its homogenization $h$, we have
\[
  h(X_0, X) \neq X_0 f(X_0^{-\frac{1}{2}}X X_0^{-\frac{1}{2}}) X_0,
  \quad X\in (\mathbb{SR}^{n})^m,\ X_0 \in \mathbb{SR}^{n} ~{\rm is ~ invertible}.
  \]
Thus, the S-lemma for general quadratic nonhomogeneous
polynomials in noncommutative cases is still unknown and left for
future research.

 }

\vspace{15pt}
\noindent{\bfseries Acknowledgments:}
 We would also like to acknowledge many valuable  comments and suggestions from  Ke Ye and  Jianting Yang.

\bibliographystyle{amsplain}
	\bibliography{yan}

\providecommand{\bysame}{\leavevmode\hbox to3em{\hrulefill}\thinspace}
\providecommand{\MR}{\relax\ifhmode\unskip\space\fi MR }
\providecommand{\MRhref}[2]{%
  \href{http://www.ams.org/mathscinet-getitem?mr=#1}{#2}
}
\providecommand{\href}[2]{#2}
\begin{thebibliography}{10}

\bibitem{Arverson69}
William~B Arveson, \emph{Subalgebras of {C}*-algebras}, Acta Mathematica
  \textbf{123} (1969), 141--224.

\bibitem{AA2001}
Aharon Ben-Tal and Arkadi Nemirovski, \emph{Lectures on modern convex
  optimization - analysis, algorithms, and engineering applications}, Society
  for Industrial and Applied Mathematics, 2001.

\bibitem{Bochnak2003}
Jacek Bochnak, Michel Coste, and Marie-Fran{\c{c}}oise Roy, \emph{Real
  algebraic geometry}, vol.~36, Springer Science \& Business Media, 2013.

\bibitem{MDC1975}
Man-Duen Choi, \emph{Completely positive linear maps on complex matrices},
  Linear Algebra and its Applications \textbf{10} (1975), no.~3, 285--290.

\bibitem{Dines1941}
Lloyd~L. Dines, \emph{On the mapping of quadratic forms}, Bulletin of the
  American Mathematical Society \textbf{47} (1941), 494--498.

\bibitem{H2002}
J.~William Helton, \emph{``{Positive}" noncommutative polynomials are sums of
  squares}, Annals of Mathematics. Second Series \textbf{2} (2002), 675--694.

\bibitem{HKIMS2010}
J.~William Helton, Igor Klep, and Scott McCullough, \emph{The matricial
  relaxation of a linear matrix inequality}, Mathematical Programming
  \textbf{138} (2010), 401--445.

\bibitem{JIS2012}
J.~William Helton, Igor Klep, and Scott McCullough, \emph{The convex
  {Positivstellensatz} in a free algebra}, Advances in Mathematics \textbf{231}
  (2012), no.~1, 516--534.

\bibitem{HKM2016}
\bysame, \emph{Matrix convex hulls of free semialgebraic sets}, Transactions of
  the American Mathematical Society \textbf{368} (2016), no.~5, 3105--3139.

\bibitem{HMS2004convex}
J.~William Helton and Scott McCullough, \emph{Convex noncommutative polynomials
  have degree two or less}, SIAM Journal on Matrix Analysis and Applications
  \textbf{25} (2004), no.~4, 1124--1139.

\bibitem{HMS2004}
J.~William Helton and Scott Mccullough, \emph{A {Positivstellensatz} for
  non-commutative}, Transactions of the American Mathematical Society
  \textbf{356} (2004), no.~9, 3721--3737.

\bibitem{ML2012}
Monique Laurent and Frank Vallentin, \emph{Semidefinite optimization}, Lecture
  Notes, 2014.

\bibitem{MC2001}
Scott McCullough, \emph{Factorization of operator-valued polynomials in several
  non-commuting variables}, Linear Algebra and Its Applications \textbf{326}
  (2001), 193--203.

\bibitem{paulsen_2003}
Vern Paulsen, \emph{Completely bounded maps and operator algebras}, Cambridge
  Studies in Advanced Mathematics, Cambridge University Press, 2003.

\bibitem{Paulsen87}
Vern~I. Paulsen, \emph{Completely bounded maps and dilations}, John Wiley \&
  Sons, Inc., USA, 1987.

\bibitem{PITT2007}
Imre P{\'o}lik and Tam{\'a}s Terlaky, \emph{A survey of the {S}-lemma}, SIAM
  Review \textbf{49} (2007), no.~3, 371--418.

\bibitem{RTR1970}
R.~Tyrrell Rockafellar, \emph{Convex analysis}, Princeton University Press,
  1970.

\bibitem{VA1971}
V.~A. Yakubovic, \emph{S-procedure in nonlinear control theory}, Vestnik
  Leningrad Univ \textbf{1} (1971), 62--77.

\bibitem{Yuan1990}
Y.~Yuan, \emph{On a subproblem of trust region algorithms for constrained
  optimization}, Mathematical Programming \textbf{47} (1990), 53--63.

\bibitem{Zar2017}
Alja{\v z} Zalar, \emph{Operator {Positivstellens{\"a}tze} for noncommutative
  polynomials positive on matrix convex sets}, Journal of Mathematical Analysis
  and Applications \textbf{445} (2017), no.~1, 32--80.

\end{thebibliography}

\end{document}